\documentclass{article}

\usepackage[utf8]{inputenc}
\usepackage{amsthm,amsmath,amssymb}
\usepackage{tikz}
\usepackage[square,numbers,sort&compress]{natbib}

\usepackage[top=30mm,bottom=30mm,left=30mm,right=30mm]{geometry}

\newtheorem{thm}{Theorem}

\newtheorem{lem}[thm]{Lemma}

\newtheorem{prop}[thm]{Proposition}
\newtheorem{definition}[thm]{Definition}

\newcommand{\R}{\mathbb{R}}
\newcommand{\Z}{\mathbb{Z}}
\newcommand{\zerovec}{\mathbb{O}}
\newcommand{\conv}{\operatorname{conv}}
\newcommand{\verts}{\operatorname{vert}}
\newcommand{\grad}{\partial}
\newcommand{\interior}{\operatorname{int}}
\newcommand{\relint}{\operatorname{relint}}
\newcommand{\bd}{\operatorname{bd}}

\newcommand{\F}{\mathcal{F}}

\title{Optimality certificates for convex minimization\\ and Helly numbers}
\author{Amitabh Basu \and Michele Conforti \and Gérard Cornuéjols \and Robert Weismantel \and Stefan Weltge}

\begin{document}

\maketitle

\begin{abstract}
We consider the problem of minimizing a convex function over a subset of $\mathbb{R}^n$ that is not necessarily convex (minimization of a convex function over the integer points in a polytope is a special case).
We define a family of duals for this problem and show that, under some natural conditions, strong duality holds for a dual problem in this family that is more restrictive than previously considered duals.
\end{abstract}
% ----------------------------------------------------------------------------------------------------------------------

\section{Introduction}

Insights obtained through duality theory have spawned efficient optimization algorithms (combinatorial and numerical) which simultaneously work on a pair of primal and dual problems. Striking examples are Edmonds' seminal work in combinatorial optimization, and interior-point algorithms for numerical/continuous optimization.

Compared to duality theory for continuous optimization, duality theory for mixed-integer optimization is still underdeveloped. Although the linear case has been extensively studied, see, e.g.,~\cite{blair1984constructive,blair1985constructive,wolsey1981b,wolsey1981integer}, nonlinear integer optimization duality was essentially unexplored until recently. An important step was taken by Morán et al. for conic mixed-integer problems~\cite{moran2012strong}, followed up by Baes et al.~\cite{baes2015duality} who presented a duality theory for general convex mixed-integer problems.
The approach taken by Moran et al. was essentially algebraic, drawing on the theory of subadditive functions.
Baes et al. took a more geometric viewpoint and developed a duality theory based on lattice-free polyhedra.
We follow the latter approach.

\medskip

\noindent
Given $ S \subseteq \R^n $ and a convex function $ f : \R^n \to \R $, we consider the problem
\begin{equation}
    \label{eqOriginalProblem}
    \inf_{s \in S} f(s).
\end{equation}
We describe a geometric dual object that can be used to certify optimality of a solution to~\eqref{eqOriginalProblem}.  For simplicity, let us consider the situation when the infimum of $f$ over $\R^n$ and over $S$ is attained, and let $x_0 \in \arg\inf_{x\in \R^n}f(x)$. We say that a closed set $C$ is an {\em $S$-free neighborhood of $x_0$} if $x_0 \in \interior(C)$ and $\interior(C) \cap S = \emptyset$. Using the convexity of $f$, it follows that for any $\bar s \in S$ and any $C$ that is an $S$-free neighborhood of $x_0$, the following holds: \begin{equation}\label{eq:weak-duality} f(\bar s) \geq \inf_{z\in \bd(C)} f(z) =: L(C),\end{equation} where $\bd(C)$ denotes the boundary of $C$ (to see this, consider the line segment connecting $\bar s$ and $x_0$ and a point at which this line segment intersects $\bd(C)$).  Thus, an $S$-free neighborhood of $x_0$ can be interpreted as a ``dual object" that provides a {\em lower bound} of the type~\eqref{eq:weak-duality}.  As a consequence, the following is true.

\begin{prop}[Strong duality]\label{prop:strong-duality} If there exist $\bar s\in S$ and $C\subseteq \R^n$ that is an $S$-free neighborhood of $x_0$, such that equality holds in~\eqref{eq:weak-duality}, then $\bar s$ is an optimal solution to~\eqref{eqOriginalProblem}.
\end{prop}
\smallskip

This motivates the definition of a dual optimization problem to~\eqref{eqOriginalProblem}. For any family $\F$ of $S$-free neighborhoods of $x_0$, define the $\F$-dual of~\eqref{eqOriginalProblem} as
\begin{equation}\label{pro:dual-unconstrained}
\sup_{C \in \mathcal{F}} L(C).
\end{equation}

Assuming very mild conditions on $S$ and $f$ (e.g., when $S$ is a closed subset of $\R^n$ disjoint from $\arg\inf_{x\in\R^n}f(x)$), it is straightforward to show that if $\F$ is the family of \emph{all} $S$-free neighborhoods of $x_0$, then strong duality holds, i.e., there exists $\bar s\in S$ and $C \in \F$ such that the condition in Proposition~\ref{prop:strong-duality} holds. However, the entire family of $S$-free neighborhoods is too unstructured to be useful as a dual problem. Moreover, the inner optimization problem~\eqref{eq:weak-duality} of minimizing on the boundary of $C$ can be very hard if $C$ has no structure other than being $S$-free. Thus, we would like to \textit{identify subfamilies $\F$ of $S$-free neighborhoods that still maintain strong duality, while at the same time, are much easier to work with inside a primal-dual framework.} We list below three subclasses that we expect to be useful in this line of research. First, we need the concept of a {\em gradient polyhedron}:

\begin{definition}\label{def:grad-poly}
Given a set of points $z_1, \ldots, z_k \in \R^n$, $$Q := \{x\in \R^n: \langle a_i, x-z_i \rangle \leq 0,\;\; i=1, \ldots, k\}$$ is said to be a {\em gradient polyhedron} of $z_1, \ldots, z_k$ if for every $i=1, \ldots, k$, $a_i \in \partial f(z_i)$, i.e., $a_i$ is a subgradient of $f$ at $z_i$.
%For a set $X\subseteq \R^n$, we will also use the notation $G(X)$ to mean the gradient polyhedron of the points in $X$.
\end{definition}
%\medskip
We consider the following families.
\begin{enumerate}
%\item Define $\F_{\conv}$ to be the family of {\em convex} $S$-free neighborhoods of $x_0$.
\item[--] The family $\F_{\max}$ of maximal convex $S$-free neighborhoods of $x_0$, i.e., those $S$-free neighborhoods that are convex, and are not strictly contained in a larger convex $S$-free neighborhood.
\item[--] The family $\F_{\grad}$ of convex $S$-free neighborhoods that are also gradient polyhedra for some finite set of points in $\R^n$.
\item[--] The family $\F_{\grad,S}$ of convex $S$-free neighborhoods that are also gradient polyhedra for some finite set of points in $S$.
%\item Define $\F_{\max-\grad} = \F_{\max} \cap $ be the family of
\end{enumerate}

We propose the above families so as to leverage a recent surge of activity analyzing their structure; %\cite{bcccz,bck,averkov-basu-lifting,bbcm,bccz08222222,bccz,bccz2,basu-paat-lifting,BorCor,conforti2013cut,dey1,dey2,delpia-4OR,dw,dw2,dey2010constrained,dey2011maximal,averkov2013maximal};
the surveys~\cite{basu2015geometric} and Chapter 6 of~\cite{conforti2014integer} provide good overviews and references for this whole line of work. This well-developed theory provides powerful mathematical tools to work with these families. As an example, this prior work shows that for most sets  $S$ that occur in practice (which includes the integer and mixed-integer cases), the family $\F_{\max}$ only contains polyhedra. This is good from two perspectives: \begin{itemize}\vspace{-5pt}\item[--] polyhedra are easier to represent and compute with than general $S$-free neighborhoods, \vspace{-5pt}\item[--] the inner optimization problem~\eqref{eq:weak-duality} of computing $L(C)$ becomes the problem of solving finitely many continuous convex optimization problems, corresponding to the facets of $C$.\end{itemize} %In addition, several results have been proved which bound the number of facets of $C$, which guarantees that the cost of these inner optimization problems will not be arbitrarily high.

Of course, the first question to settle is whether these three families actually enjoy strong duality, i.e., do we have strong duality between~\eqref{eqOriginalProblem} and the $\F_{\max}$-dual, $\F_{\grad}$-dual and $\F_{\grad,S}$-dual? It turns out that the main result in~\cite{baes2015duality} shows that for the mixed-integer case, i.e., $S = C\cap(\Z^{n_1}\times \R^{n_2})$ for some convex set $C$, the $\F_{\grad}$-dual enjoys strong duality under conditions of the Slater type from continuous optimization. It is not hard to strengthen their result to also show that the $\F_{\max} \cap \F_{\grad}$-dual is a strong dual, under some additional assumptions.

In this paper, we give conditions on $S$ and $f$ such that strong duality holds for the dual problem \eqref{pro:dual-unconstrained} associated with $\F_{\max}\cap \F_{\grad}\cap\F_{\grad,S}$. Below we give an explanation as to why this family is very desirable. If these conditions on $S$ and $f$ are met, our result is stronger than Baes et al.~\cite{baes2015duality}. For example, when $S$ is the set of integer points in a compact convex set and $f$ is any convex function, our certificate is a stronger one. However, our conditions on $S$ and $f$ do not cover certain mixed-integer problems; whereas, the certificate from Baes et al. still exists in these settings. Nevertheless, it can be shown that in such situations, a strong certificate like ours does not necessarily exist.

%These sets $S$ are not contained and do not contain the sets considered by Baes et al.

%We establish a stronger result in Theorem~\ref{thmMain} below which implies that for very general sets $S \subset \R^n$, strong duality holds for all three families simultaneously; in fact it shows strong duality for the family $\F_{\max}\cap\F_{\grad}\cap\F_{\grad,S} = \F_{\max}\cap\F_{\grad,S}$.

%

\begin{definition} A \emph{strong optimality certificate of size $ k $} for \eqref{eqOriginalProblem} is a set of points $ z_1,\dotsc,z_k
\in S $ together with subgradients $ a_i \in \grad f(z_i) $ such that
\begin{align}
    \label{strongProperty1}
    & Q := \{ x \in \R^n : \langle a_i, x - z_i \rangle \le 0, \ i=1,\dotsc,k \} \text{ is S-free},\\
    \label{strongProperty2}
    & \langle a_i, z_j - z_i \rangle < 0 \text{ for all } i \ne j.
\end{align}
\end{definition}
%
%In the above, a closed convex set $ C \subseteq \R^n $ is an \emph{$ S $-free} set if $\interior(C)\cap S=\emptyset$
%(where for $C\in \R^n$, we denote with $\interior(C)$ its interior).
%\medskip

Recall that $ a \in \grad f(z) $ if $ f(x) \ge f(z) + \langle a,x - z\rangle $ holds for all $x\in \R^n $.
Since $ Q $ is $ S $-free, for every $ s \in S $ there is some $ i \in [k] $ such that $ \langle a_i,s - z_i \rangle \ge
0 $ and hence $ f(s) \ge f(z_i) $.
Thus, Property~\eqref{strongProperty1} implies that $ \min_{s\in S} f(s) = \min_{i\in [k]} f(z_i) $ holds.
In other words, given a strong optimality certificate, we can compute~\eqref{eqOriginalProblem} by simply evaluating $
f(z_1),\dotsc,f(z_k) $. This implies that if a strong certificate exists, then the infimum of $f$ over $S$ is attained.

In order to verify that $ z_1,\dotsc,z_k $ together with $ a_1,\dotsc,a_k $ form a strong optimality certificate, one
has to check whether the polyhedron $ Q $ is $ S $-free.
Deciding whether a \emph{general} polyhedron is $ S $-free might be a difficult task.
However, Property~\eqref{strongProperty2} ensures that $ Q $ is \emph{maximal} $ S $-free, i.e., $ Q $ is not properly
contained in any other $ S $-free closed convex set:
Indeed, Property~\eqref{strongProperty2} implies that $ Q $ is a full-dimensional polyhedron and that  $\{ x \in Q :
\langle a_i,x \rangle = 0 \} $ is a facet of $ Q $ containing $ z_i \in S $ in its relative interior for every $ i \in
[k] $.
Since every closed convex set $ C $ that properly contains $ Q $ contains the relative interior of at least one
facet of $ Q $ in its interior, $ C $ cannot be $ S $-free.

For particular sets $S$, the properties of $ S $-free sets that are maximal have been extensively studied and are much
better understood than general $ S $-free sets.
For instance, if $S = (\R^d \times \Z^n) \cap C $ where $ C $ is a closed convex subset of $\R^{n+d}$, maximal $S$-free
sets are polyhedra with at most $2^n$ facets~\cite{dey2011maximal}.
In particular, if $ S = \Z^2 $ the characterizations in~\cite{Hurkens90,dw2008} yield a very simple
algorithm to detect whether a polyhedron is maximal $ \Z^2 $-free.

\medskip

In order to state our main result, we need the notion of the \emph{Helly number} $ h(S) $ of the set $S$, which is the
largest number $ m $ such that there exist convex sets $ C_1,\dotsc,C_m \subseteq \R^n $ satisfying
\begin{equation}
    \label{eqHellyDefinition}
    \bigcap_{i \in [m]} C_i \cap S = \emptyset \quad \text{ and } \bigcap_{i \in [m] \setminus \{j\}} C_i \cap S \ne
    \emptyset \text{ for every } j \in [m].
\end{equation}

\begin{thm}
    \label{thmMain}
    Let $ S \subseteq \R^n $ and $ f : \R^n \to \R $ be a convex function such that
    \begin{itemize}
        \item[(i)] $ \zerovec \notin \grad f(s) $ for all $ s \in S $,
        \item[(ii)] $ h(S) $ is finite, and
        \item[(iii)] for every polyhedron $ P \subseteq \R^n $ with $ \interior(P) \cap S \ne \emptyset $ there exists
        an $ s^\star \in \interior(P) \cap S $ with $ f(s^\star) = \inf_{s \in \interior(P) \cap S} f(s) $.
    \end{itemize}
    Then there exists a strong optimality certificate of size at most $ h(S) $.
\end{thm}
Let us first comment on the assumptions in Theorem~\ref{thmMain}.
If $ \zerovec \in \grad f(s^\star) $ for some $s^\star \in S$, then $ s^\star $ is an optimal
solution to~\eqref{eqOriginalProblem}, as well as to its continuous relaxation over $\R^n$. An easy certificate of optimality in this case is the subgradient $\zerovec$.
%
%It turns out that a strong optimality certificate cannot exist if $ \zerovec \in \grad f(s^\star) $ for some $ s^\star\in S $.
%
%However, in this case the fact that $ \zerovec \in \grad f(s^\star) $ shows already that $ s^\star $ is an optimal solution to~\eqref{eqOriginalProblem}, as well as to its continuous relaxation over $\R^n$.
%
A quite general situation in which~(ii) is always satisfied is the case $ S = (\R^d \times \Z^n) \cap C $ where $ C
\subseteq \R^{d+n} $ is a closed convex set.
In this situation, one has $ h(S) \le 2^n (d+1) $. The characterization of closed sets $S$ for which $h(S)$ is finite has received a lot of attention, see, e.g.,~\cite{averkov2013maximal}.
Finally, note that~(iii) implies that the minimum in~\eqref{eqOriginalProblem} actually exists.
As an example, (iii) is fulfilled whenever $ S $ is discrete (every bounded subset of $ S $ is finite) and the set $ \{
x \in \R^n : f(x) \le \alpha \} $ is bounded and non-empty for some $ \alpha \in \R $ (implying that the set is actually
bounded for every $ \alpha \in \R $). This latter condition is satisfied, e.g., when $f$ is strictly convex and has a minimizer. Another situation where (iii) is satisfied is when $S$ is a finite set, e.g., $S = C \cap \Z^n$ where $C$ is a compact convex set.

Also, if conditions (i) and (ii) hold, but (iii) does not hold, a strong optimality certificate may not exist. For example, consider $S = \{x \in \Z^2: \sqrt{2}{x_1} - x_2 \geq 0,\;\;x_1 \geq \frac12,\;\;x_2 \geq 0\}$ and $f(x) = \sqrt{2}x_1 - x_2$. In this case, no strong optimality certificate can exist, as the infimum of $f$ over $S$ is $0$, but it is not attained by any point in $S$.

%We point out that if $S$ is not discrete, it is unlikely that $(iii)$ will hold. %In this respect, our result is weaker than Baes et al.~\cite{baes2015duality} because their certificate applies to the mixed-integer case $S = \R^d\times \Z^n$. On the other hand, our result handles other $S$ that are not considered by Baes et al.

% ----------------------------------------------------------------------------------------------------------------------

\section{Proof of Theorem~\ref{thmMain}}

We make use of the following observation.
Let $ \conv(\cdot) $ denote the convex hull and $ \verts(P) $ denote the set of vertices of a polyhedron $ P $.
\begin{lem}
    \label{lemHelly}
    Let $ S \subseteq \R^n $ and $ V \subseteq S $ finite such that $ V = \conv(V) \cap S = \verts(\conv(V)) $.
    Then we have $ |V| \le h(S) $.
\end{lem}
\begin{proof}
    Let $ V = \{ v_1,\dotsc,v_m \} $ and for every $ i \in [m] $ let $ C_i := \conv(V \setminus \{v_i\}) $.
    Since $ V = \conv(V) \cap S = \verts(\conv(V)) $, we have $ C_i \cap S = V \setminus \{v_i\} $ for every $ i \in [m]
    $.
    Thus, $ C_1,\dotsc,C_m $ satisfy~\eqref{eqHellyDefinition} and hence $ m \le h(S) $.
\end{proof}

\noindent
We are ready to prove Theorem~\ref{thmMain}.
Let us consider the following algorithm (in fact, we will see that this is indeed a finite algorithm):
\begin{align}
    \nonumber
    & Q_0 \leftarrow \R^n, \, k \leftarrow 1 \\
    \nonumber
    & while \interior(Q_{k-1}) \cap S \ne \emptyset: \\
    \label{algComputeMinimum}
    & \qquad t_k \leftarrow \min \{ f(s) : s \in \interior(Q_{k-1}) \cap S \} \\
    \nonumber
    & \qquad C_k \leftarrow \{ x \in \R^n : f(x) \le t_k \} \\
    \label{algMaxDim}
    & \qquad z_k \leftarrow \text{any } s \in \interior(Q_{k-1}) \cap S \text{ with } f(s) = t_k \text{ such that }
    \dim(F_{C_k}(s)) \text{ is largest possible } \\
    \nonumber
    & \qquad a_k \leftarrow \text{any point in } \relint(\grad f(z_k)) \\
    \nonumber
    & \qquad Q_{k} \leftarrow \{ x \in Q_{k-1} : \langle a_k,x - z_k \rangle \le 0 \} \\
    \nonumber
    & \qquad k \leftarrow k + 1
\end{align}
In the above, $ \relint(\cdot) $ denotes the relative interior and $ \dim(\cdot) $ the
affine dimension.
For a closed convex set $ C \subseteq \R^n $ and a point $ p \in C $ we denote by $ F_C(p) $ the smallest face of $
C $ that contains $ p $.
%Note that $p\in \relint(F_C(p))$.

Remark that iteration $k$ of the algorithm can always be executed, as the set $ Q_k $ is a polyhedron and hence by
the assumption in~(iii) the minimum in~\eqref{algComputeMinimum} always exists. Furthermore, since $ a_k \in
\relint(\grad f(z_k))$ we have
\begin{equation}
    \label{eqFace}
    F_k := F_{C_k}(z_k) = \{ x \in C_k : \langle a_k,x-z_k \rangle = 0 \}
\end{equation}

\smallskip

\noindent
\textit{Claim 1: For every $ k $ we have that $ \langle a_i, z_j - z_i \rangle < 0 $ holds for all $ i,j \le k $
with $ i \ne j $.}

\smallskip

\noindent
Let $ k \ge 2 $ and assume that the claim is satisfied for all $ i,j \le k - 1 $, $ i \ne j $.
Since $ z_k \in \interior( Q_{k-1} )$ and $ a_i \ne \zerovec $ by assumption (i), we have that $ \langle a_i,z_k -
z_i \rangle < 0 $ for every $ i < k $.

It remains to show that $ \langle a_k,z_i-z_k \rangle < 0 $ for every $ i < k $.
Since $a_k\in \grad f(z_k)$, we have that $ \langle a_k,z_i-z_k \rangle \le f(z_i) - f(z_k)$
and for $ i < k $ by \eqref{algComputeMinimum} we have $ f(z_i) \le f(z_k) $.
Therefore $ \langle a_k,z_i-z_k \rangle \le 0 $  and if $ \langle a_k,z_i-z_k \rangle  = 0 $, then $ f(z_i) =
f(z_k) $.
Assume this is the case. Since $ \langle a_i,z_k-z_i \rangle < 0 $ we have $ z_k \notin F_i $ and in particular
\begin{equation}
    \label{eqFacesNotEqual}
    F_i \ne F_k.
\end{equation}
%
%Note that we have $ \langle a_k,z_i - z_k \rangle \le f(z_i) - f(z_k) = 0 $.
%
By~\eqref{eqFace} this means that $ z_i \in F_k $ holds.
Since $ F_i $ is the smallest face that contains $ z_i $, this implies $ F_i \subseteq F_k $.
By~\eqref{algMaxDim}, we have that $ \dim(F_i) \ge \dim(F_k) $ and thus $ F_i = F_k $, a contradiction
to~\eqref{eqFacesNotEqual}.

\smallskip

\noindent
\textit{Claim 2: For every $ k $ we have that $ V := \{ z_1,\dotsc,z_k \} $ satisfies $ V = conv(V) \cap S =
\verts(\conv(V)) $.}

\smallskip

\noindent
It is easy to see that Claim~1 implies $ V = \verts(\conv(V)) $.
For the sake of contradiction, assume there exists some $ s \in (\conv(V) \setminus V) \cap S $.
By Claim~1, we have $ s \in \interior(Q_{k}) $. Therefore by~\eqref{algComputeMinimum} we have $ f(s) \ge t_k $.
Since $ f $ is convex and $ s \in \conv(V) $, this implies $ f(s) = t_k $.
Let $ a \in \relint(\grad f(s)) $ and consider $ F := F_{C_k}(s) = \{ x \in C_k: \langle a, z_i - s \rangle = 0\} $.
Since $ V \subseteq C_k $, we have that $ z_i \in F $ for at least one $ i \in [k] $.
Due to $ \langle a,z_i - s \rangle \le f(z_i) - f(s) $ we must have $ f(z_i) = t_k $ and hence $ F_i \subseteq F $.
By~\eqref{algMaxDim}, we further have $ \dim(F_i) \ge \dim(F) $, which shows $ F_i = F $.
However, by Claim~1 we have $ z_j \notin F_i $ for all $ j \ne i $ and hence $ s \notin F_i $, a contradiction since $ s
\in F $.

\smallskip

\noindent
\textit{Claim 3: The algorithm stops after at most $ h(S) $ iterations and $ Q := Q_k $ is $ S $-free.}

\smallskip

\noindent
Note that the set $ V := \{z_1,\dotsc,z_k\} $ becomes larger in every iteration.
By Claim~2 and Lemma~\ref{lemHelly} we must have $ k \le h(S) $ and hence the algorithm stops after at most $ h(S) $
iterations.
Since the algorithm stops if and only if $ Q_k $ is $ S $-free, this proves the claim.
\hfill $ \qedsymbol $

\section*{Acknowledgements}
This work was performed during a stay of the first three authors at ETH Zürich.
Their visit had been funded by the Research Institute of Mathematics (FIM).
Furthermore, part of this work was supported by NSF grant 1560828 and ONR grant 00014-15-12082.
We are grateful for this support.

\bibliographystyle{plain}
\bibliography{references}

\end{document}